\documentclass[final]{siamart1116}

\usepackage{amsfonts}
\usepackage{graphicx}

\usepackage{algorithmic}
\Crefname{ALC@unique}{Line}{Lines}

\ifpdf
  \DeclareGraphicsExtensions{.eps,.pdf,.png,.jpg}
\else
  \DeclareGraphicsExtensions{.eps}
\fi

\usepackage{xspace}          
\usepackage{amssymb}         

\numberwithin{theorem}{section}

\newcommand{\TheTitle}{Listing Words in Free Groups} 
\newcommand{\TheAuthors}{C. Ramsay}

\headers{\TheTitle}{\TheAuthors}

\title{{\TheTitle}
}

\author{
  Colin Ramsay\thanks{School of Information Technology and Electrical
  Engineering, The University of Queensland, Australia
 (\email{uqcramsa@uq.edu.au}).}
}

\newcommand{\qb}[1]{\textbf{#1}}
\newcommand{\qs}[1]{\textsc{#1}}
\newcommand{\qt}[1]{\texttt{#1}}

\newcommand{\qq}{\quad}

\newcommand{\za}{\textbf{begin}\xspace}
\newcommand{\zz}{\textbf{end}\xspace}

\newcommand{\zi}{\textbf{if}\xspace}
\newcommand{\zt}{\textbf{then}\xspace}
\newcommand{\ze}{\textbf{else}\xspace}

\newcommand{\zr}{\textbf{return}\xspace}

\ifpdf
\hypersetup{
  pdftitle={\TheTitle},
  pdfauthor={\TheAuthors}
}
\fi

\begin{document}

\maketitle

\begin{abstract}
Lists of equivalence classes of words under rotation or rotation plus reversal 
  (i.e., necklaces and bracelets) have many uses, and efficient algorithms for
  generating these lists exist.
In combinatorial group theory elements of a group are typically written as words
  in the generators and their inverses, and necklaces and bracelets correspond
  to conjugacy classes and relators respectively.
We present algorithms to generate lists of freely and cyclically reduced
  necklaces and bracelets in free groups.
Experimental evidence suggests that these algorithms are CAT -- that is, they
  run in constant amortized time.
\end{abstract}



\section{Introduction}

Given an ordered alphabet of size $k$, a \emph{necklace} of length $n$ is the 
  lexicographically least element of an equivalence class of $k$-ary strings
  of length $n$ under rotation.
A word is called a \emph{prenecklace} if it is the prefix of some necklace.
An aperiodic necklace is called a \emph{Lyndon word}.
A \emph{bracelet} of length $n$ is the lexicographically least element of an 
  equivalence class of $k$-ary strings of length $n$ under string rotation and
  string reversal.

For a fixed $k$ the number of necklaces (also Lydon words, prenecklaces and 
  bracelets) grows exponentially with the length.
See, for example, \cite{CRSSM00, Saw01} for exact counts and bounds.
So generating a complete list of the length $n$ necklaces takes exponential 
  time, and our goal is an algorithm where the computation (the total amount of
  change to the data structures, not including any processing of the generated
  necklaces) is proportional to the number of necklaces generated.
Such an algorithm is a constant amortized time, or CAT\kern-1pt, algorithm.

In group theory, elements of a group can be represented by strings (or 
  \emph{words}) in the group's generators and their inverses.
Symbolic algebra systems such as \qt{GAP} and \qs{Magma} \cite{GAP4,MR1484478} 
  make sophisticated testing of large numbers of examples straightforward, and
  efficient algorithms for generating complete lists of words, up to some
  equivalence, are an important part of this.
The extant enumeration algorithms do not take into account the group structure,
  and we demonstrate how they can be recast to address this.
For necklaces this process is trivial, while for bracelets we need modify the
  reversal checking code materially.

The remainder of this paper is organized as follows.
\Cref{sec:back} gives some background material on necklaces and bracelets and on
  free groups and group presentations, and discusses the analogues of
  necklaces and bracelets in groups.
\Cref{sec:neck,sec:brace}, respectively, describe our necklace and bracelet 
  listing algorithms, and we discuss our results in \Cref{sec:discuss}.
\Cref{sup:comp} describes the tests we
  performed on the running times of our algorithms.

\section{Background}\label{sec:back}

The first algorithm for generating necklaces, the FKM algorithm (due to 
  Fredericksen, Kessler and Maiorana \cite{FK86, FM78}), was proved to be CAT in
  \cite{RSW91}.
A simple recursive CAT algorithm to generate prenecklaces, necklaces and Lyndon 
  words was given in \cite{CRSSM00}, and it is this algorithm which forms the
  basis of our work.

Duval's algorithm for factoring a string, of length $n$, into Lyndon words 
  \cite{Duv83} yields an algorithm for generating the necklace of the string in 
  $O(n)$ time.
Thus a straightforward approach to generating bracelets is to generate the 
  necklaces and to reject those where the necklace of the reversal is less than
  the necklace.
However, this does not yield a CAT algorithm.
The algorithm given in \cite{Saw01} is based on the recursive algorithm of 
  \cite{CRSSM00} and maintains auxiliary data regarding the current prenecklace,
  using this to guide testing against its reversal and control the computation.
The total amount of extra work, amortized over all bracelets, is constant, so
  this bracelet generating algorithm is CAT\kern-1pt.

Given a set $S$ of $g$ symbols define the set $S^\prime = 
  S \cup \{s^{-1} : s \in S\}$.
The set of all words on $S^\prime$ is $F_g$, the \emph{free group} of 
  \emph{rank} $g$.
The group operation is concatenation, $s^{-1}$ is read as the inverse of $s$,
  and the empty word is the identity element.
Words with no substrings of the form $ss^{-1}$ or $s^{-1}s$ are called 
  \emph{freely reduced}.
Two words represent the same element of $F_g$ if and only if they are
  identical after being freely reduced.
(See \cite{LS1977,MKS66} for more details on combinatorial group theory.)

Let $w = s_1 \cdots s_\ell \in F_g$.
If $s_1$ and $s_\ell$ are not inverses of each other, then $w$ is 
  \emph{cyclically reduced}.
Given an $r \in S^\prime$ then the word $r^{-1} w r$ is the \emph{conjugate} of
  $w$ by $r$\kern-1pt, denoted $w^r$\kern-2pt.
If $r=s_1$ (resp., $s_\ell^{-1}$) and the substring $r^{-1}s_1$ (resp.,
  $s_\ell r$) is canceled from $w^r$ then the resulting word is a rotation of
  $w$ by one position.
If $r = s_1 = s_\ell^{-1}$ then canceling $r^{-1}s_1$ and $s_\ell r$ performs a
  cyclic reduction step and reduces the length of $w$ by two.
Repeated conjugation of a word may render it cyclically reduced or freely 
  reduced, rotate it arbitrarily, or increase its length arbitrarily.

Conjugation partitions the words in $F_g$ into \emph{conjugacy classes}.
Given an order on $S^\prime$\kern-2pt, we take as class representatives the 
  lexicographically least element of the freely and cyclically reduced words in
  the class.
So, in the context of $F_g$, listing the freely and cyclically reduced necklaces
  of length $\ell$ is equivalent to listing the conjugacy classes whose shortest
  words have length $\ell$.
A word which is both freely reduced and cyclically reduced will be called
  simply \emph{reduced}.
 
Reversing a word $w$ is not meaningful in $F_g$.
However reversing $w$ and then replacing each of its elements by its inverse
  generates $w^{-1}$ (i.e., freely reducing $ww^{-1}$ or $w^{-1}w$ results in 
  the empty word).
Given a set of words $R$ in $F_g$, the \emph{normal closure} $N$ of $R$ in 
  $F_g$ is the set of all words which are concatenations of conjugates of the
  words in $R$ and their inverses.
Groups are often described as quotient groups of free groups and $N$ is a
  normal subgroup of $F_g$, so the quotient $F_g / N$ describes some group $G$.
(Formally, there is a homomorphism from $F_g$ onto $G$ with kernal $N$\kern-1pt,
  see \cite{LS1977,MKS66}.)

The pair $(S,R)$ is a \emph{presentation} for $G$, written as
   $G = \langle S : R \rangle$.
The elements of $S$ are the \emph{generators} of $G$.
The words in $R$ are equal to the identity in $G$ and are called 
  \emph{relators}.
So, in the context of $F_g$, listing the reduced bracelets of length $\ell$ is
  equivalent to listing equivalence classes of possible relators of length
  $\ell$ in a presentation.

Enumerating reduced necklaces and bracelets in groups is equivalent to
  enumerating general necklaces and bracelets with forbidden substrings.
An efficient algorithm exists to enumerate necklaces with a forbidden substring 
  \cite{RS00}, however the analysis therein assumes that the substring has
  length at least three.
For substrings of length one or two, \cite{RS00} notes that ``trivial 
  algorithms can be developed''\kern-2pt.
In our case, we simply test each potential addition to the current prenecklace,
  and skip those which cannot yield a reduced necklace.

In the remainder of this paper, unless explicitly stated otherwise, we are 
  always working in the free group $F_g$.
The number of group generators will be denoted by $g$ and the word length by
  $\ell$ (both assumed positive), with the set of possible symbols in our words
  having size $k = 2g$.
From \cite[Theorems 1.1 \& 14.2]{Riv99} we have the following result.

\begin{theorem}\label{thm:X}
The number of reduced words of length $\ell$ in $F_g$ is equal to
  $$ \mathcal{C}(g,\ell) = \begin{cases}
   (2g-1)^\ell + 1,    & \text{if $\ell$ is odd;} \\
   (2g-1)^\ell + 2g-1, & \text{if $\ell$ is even.}
  \end{cases} $$
Let $\phi$ denote the Euler totient function.
Then the number of reduced necklaces of length $\ell$ in $F_g$ is equal to
  $$ \mathcal{CC}(g,\ell) = \frac{1}{\ell} \sum_{d \mid \ell}
                                               \phi(d) \mathcal{C}(g,\ell/d). $$
\end{theorem}

In the general case (i.e, not in $F_g$) it is possible for a necklace and its
  reversal to be equal, up to rotation -- consider the necklace $ababb$ and its
  reversal $bbaba$.
However, in $F_g$ a reduced necklace cannot be equal to its inverse or any of
  its inverse's rotations.
More generally, we have the following result.

\begin{lemma}\label{lem:Z}
Let $w$ be a freely reduced word of length $\ell > 0$ in $F_g$.
Then no conjugate of $w^{-1}$ equals $w$.
\end{lemma}

\begin{proof}
Let $w = s_1 \cdots s_\ell$, and write $1$ for the empty word and $\bar{x}$
  for $x^{-1}$\kern-2pt.

(i)
We first prove that $w \neq \bar{w}$.
If $w=\bar{w}$ then $w$ is its own inverse and so $ww = 1$.
Now put $w = \bar{u}vu$, where $u,v \in F_g$ are freely reduced and $u$ has 
  maximal length.
Since $w$ is freely reduced, $v$ is non-empty and, by $u$'s maximality, there is
  no free reduction in $vv$.
Thus $ww = \bar{u}vu\bar{u}vu = \bar{u}vvu$ is a freely reduced non-empty word,
  contradicting $ww=1$.

(ii)
Now assume that $w_r = 
  \bar{s}_k\cdots \bar{s}_1\bar{s}_\ell\cdots\bar{s}_{k+1}$, 
  $1 \leqslant k \leqslant \ell-1$, is a proper rotation of $\bar{w}$ which
  equals $w$.
If $\bar{s}_1\bar{s}_\ell = 1$ then free reduction of $w_r$ yields a word
  of length $n<\ell$, so $w \neq w_r$.
Thus $w_r$ must be freely reduced, and $w=w_r$ implies that 
  $w_1 = \bar{s}_k\cdots \bar{s}_1 = s_1 \cdots s_k = \bar{w}_1$.
However this is impossible (part (i), with $w=w_1$), so $\bar{w}$ is not a
  proper rotation of $w$.

(iii)
Now consider arbitrary conjugation of $\bar{w}$, followed by free reduction.
If this yields a word of length $n\neq\ell$, then $w\neq\bar{w}$.
If not, then part (i) or (ii) applies.
\end{proof}

Thus the set of reduced necklaces in $F_g$ of length $\ell$ can be partitioned
  into pairs, where the words in a pair are inverses (up to rotation) and are
  not equal under rotation.
So precisely one member of each pair is a bracelet, and the number of reduced
  bracelets of length $\ell$ in $F_g$ is $ \mathcal{CC}(g,\ell) / 2 $.

In some applications we may only be interested in aperiodic (or \emph{prime})
  words, and it is trivial to modify our algorithms to generate
  these (see \Cref{sec:discuss}).
From \cite[Equation (2.2)]{Coo05} we have the following result.

\begin{theorem}\label{thm:Y}
Let $\mu$ denote the M\"obius function.
Then the number of reduced prime words of length $\ell$ in $F_g$ is equal to
  $$ \tau(g,\ell) = \sum_{d \mid \ell} 
                             \mu\left(\frac{\ell}{d}\right) \mathcal{C}(g,d). $$
\end{theorem}

Thus, in $F_g$ there are $\tau(g,\ell)/\ell$ reduced necklaces and
  $\tau(g,\ell)/2 \ell$ reduced bracelets of length $\ell$ which are 
  not proper powers.

For $F_1$, with generator $z$, there are only two reduced necklaces ($z^\ell$
  and $(z^{-1})^\ell$) and one reduced bracelet ($z^\ell$) for all $\ell>0$, so
  we ignore this case and assume throughout that $g>1$.
Obviously, for $g=1$ and $\ell>1$ there are no reduced prime words.
The algorithms we give are valid for $g=1$ but they are not CAT, since it takes
  $O(\ell)$ time to set the word to $z\cdots z$ or to 
  $z^{-1} \!\cdots z^{-1}$\kern-1pt.

\section{Listing Necklaces}\label{sec:neck}

We first need to decide on the conventions we adopt for representing
  the group generators and their inverses, and how to order these $k$ symbols.
Using the integers $\pm 1$, $\dots$, $\pm g$ is straightforward and is
  convenient for generating and checking inverses.
However, it is awkward for running through the symbols in order and checking 
  symbol ordering, since we require a generator to precede its inverse and
  for $\pm 1$ to precede $+2$, etc.
Accordingly, we use the integers $0$, $\dots$, $k-1$ with the usual numeric
  ordering, where even integers $j$ denote the generators and odd integers
  $j+1$ their inverses.

\begin{algorithm}[t]
\caption{The \qt{areInv()} function}
\label{alg:a1}
\begin{algorithmic}[1]
\setcounter{ALC@unique}{0}     
\STATE{ \qb{function} \qt{areInv}( $x,y$ : integer ) \qb{returns} boolean; }
\STATE{ \za }
\STATE{ \qq \zi $x \bmod 2 = 0$ \zt }
\STATE{ \qq\qq \za \zi $x+1 = y$ \zt \zr \qs{true}; \zz; }
\STATE{ \qq \ze }
\STATE{ \qq\qq \za \zi $x = y+1$ \zt \zr \qs{true}; \zz; }
\STATE{ \qq \zr \qs{false}; }
\STATE{ \zz; }
\end{algorithmic}
\end{algorithm}

\begin{algorithm}[t]
\caption{The \qt{getInv()} function}
\label{alg:a2}
\begin{algorithmic}[1]
\setcounter{ALC@unique}{0}     
\STATE{ \qb{function} \qt{getInv}( $x$ : integer ) \qb{returns} integer; }
\STATE{ \za }
\STATE{ \qq \zi $x \bmod 2 = 0$ \zt \zr $x+1$; }
\STATE{ \qq \zr $x-1$;}
\STATE{ \zz; }
\end{algorithmic}
\end{algorithm}

To handle inverses we introduce the two utility functions \qt{areInv()} and
  \qt{getInv()} of \cref{alg:a1,alg:a2}.
These, respectively, check whether or not two symbols are inverses and return
  the inverse of a symbol.
These functions are common to both the necklace and bracelet algorithms and run
  in constant time.
We separate out these functions for simplicity -- in practice they can be
  compiled as inline functions or replaced by macros.

Our recursive necklace generation procedure \qt{genNeck()} in \cref{alg:a3} is
  now a simple modification of \cite[Algorithm 2.1]{CRSSM00}, with the \zi 
  statements at \cref{line3.9,line3.13} ensuring that the prenecklaces remain
  freely reduced and that the final necklaces are cyclically reduced.
For clarity we do not use the guard value $a_0$ of \cite{CRSSM00} and 
  instead use the wrapper code of \cref{alg:a4}.
This explicitly sets $a_1$ (recording its inverse in $aoi$ to facilitate
  cyclic reduction checking) and ensures that the prenecklaces at entry to
  \qt{genNeck()} are non-empty.

\begin{algorithm}[t]
\caption{The \qt{genNeck()} procedure}
\label{alg:a3}
\begin{algorithmic}[1]
\setcounter{ALC@unique}{0}     
\STATE{ \qb{procedure} \qt{genNeck}( $t,p$ : integer ); }
\STATE{ \qb{local} $j$ : integer; }
\STATE{ \za }
\STATE{ \qq \zi $t > \ell$ \zt \za }
\STATE{ \qq\qq \zi $\ell \bmod p = 0$ \zt process necklace $a_1 \cdots a_\ell$; }
\STATE{ \qq\qq \zr; }
\STATE{ \qq \zz;}
\STATE{ \qq $j := a_{t-p}$; }
\STATE\label{line3.9}{ \qq \zi \qb{not} \qt{areInv}( $a_{t-1},j$ ) \qb{and}
                                    ( $t < \ell$ \qb{or} $j \ne aoi$ ) \zt \za }
\STATE{ \qq\qq $a_t := j$; \ \qt{genNeck}( $t+1,p$ ); }
\STATE{ \qq \zz; }
\STATE{ \qq \qb{for} $j$ \qb{from} $a_{t-p}+1$ \qb{to} $k-1$ \qb{do} \za }
\STATE\label{line3.13}{ \qq\qq \zi \qb{not} \qt{areInv}( $a_{t-1},j$ ) \qb{and}
                                    ( $t < \ell$ \qb{or} $j \ne aoi$ ) \zt \za }
\STATE{ \qq\qq\qq $a_t := j$; \ \qt{genNeck}( $t+1,t$ ); }
\STATE{ \qq\qq \zz; }
\STATE{ \qq \zz; }
\STATE{ \qq \zr; }
\STATE{ \zz; }
\end{algorithmic}
\end{algorithm}

\begin{algorithm}[t]
\caption{The necklace code wrapper}
\label{alg:a4}
\begin{algorithmic}[1]
\setcounter{ALC@unique}{0}     
\STATE{ \qb{global} $\ell,g,k,aoi$ : integer; \ $a$ : integer array; }
\STATE{ \qb{local} $j$ : integer; }
\STATE{ $g :=$ number of generators; \ $\ell :=$ word length; }
\STATE{ $k := 2g$; }
\STATE{ \qb{for} $j$ \qb{from} $0$ \qb{to} $k-1$ \qb{do} \za }
\STATE{ \qq $a_1 := j$; }
\STATE{ \qq $aoi :=$ \qt{getInv}( $j$ ); }
\STATE{ \qq \qt{genNeck}( 2,1 ); }
\STATE{ \zz; }
\end{algorithmic}
\end{algorithm}

\section{Listing Bracelets}\label{sec:brace}

Our bracelet algorithm is inspired by that in \cite{Saw01}, with reversal of a 
  prenecklace being replaced by word inversion (i.e., reversal and
  element-to-inverse mapping).
The recursive bracelet generation procedure \qt{genBrace()} and its wrapper code
  are given in \cref{alg:a5,alg:a6}, while \cref{alg:a7} is the \qt{checkInv()}
  function for comparing a prenecklace with its inverse.
The \qt{genBrace()} procedure is an augmented version of \qt{genNeck()}, with
  the additional code checking each prenecklace (i.e., the putative
  ``prebracelets'') against its inverse and rejecting those that cannot yield
  bracelets.
Thus each rotation of the inverse of the final words is tested, and necklaces
  which are not bracelets are not generated.

The use of inverses as opposed to reversals actually results in a somewhat 
  simpler algorithm compared with that in \cite{Saw01}.
Firstly, note that our order is chosen so that a generator immediately precedes
  its inverse.
This implies that bracelets cannot start with a generator inverse and so these
  can be skipped in the wrapper code.
Secondly, although the \qt{checkInv()} function can return ``equal'' as well as 
  ``less than'' or ``greater than'' (as does the \qt{CheckRev()} function in 
  \cite{Saw01}), by \cref{lem:Z} it never does so since our preneckaces are
  always freely reduced.
This simplifies the code in the \qt{genBrace()} procedure, which needs only four
  parameters compared with the six of \qt{GenBracelets()} in \cite{Saw01}.

The $t$ and $p$ arguments of \qt{genBrace()} are, respectively, the index of the
  next
  position in the array $a$ and the length of the longest prefix of $a$ that is
  a Lyndon word.
The $u$ and $v$ arguments are, respectively, the number of copies of $a_1$ at
  the start of $a$ and the number of copies of $a_1^{-1} = aoi$ at the end of
  $a$, with the initial value of $v$ saved in the
  local variable $vv$.
The code at \crefrange{line5.9}{line5.14} and \cref{line5.22,line5.25} adjusts
  $u$ and $v$ as necessary for the next call to \qt{genBrace()} (if any), using
  the current value of $j$ (i.e., the next potential value for $a_t$) and the
  current word length $t-1$.

\begin{algorithm}[t]
\caption{The \qt{genBrace()} procedure}
\label{alg:a5}
\begin{algorithmic}[1]
\setcounter{ALC@unique}{0}     
\STATE{ \qb{procedure} \qt{genBrace}( $t,p,u,v$ : integer ); }
\STATE{ \qb{local} $j,vv := v$ : integer; }
\STATE{ \za }
\STATE{ \qq \zi $t > \ell$ \zt \za }
\STATE{ \qq\qq \zi $\ell \bmod p = 0$ \zt process bracelet $a_1 \cdots a_\ell$; }
\STATE{ \qq\qq \zr; }
\STATE{ \qq \zz;}
\STATE{ \qq $j := a_{t-p}$; }
\STATE\label{line5.9}{ \qq \zi $j = a_1$ \zt }
\STATE{ \qq\qq \za $v := 0$; \ \zi $u = t-1$ \zt $u := u+1$; \zz; }
\STATE{ \qq \ze \zi $j = aoi$ \zt }
\STATE{ \qq\qq $v := v+1$; }
\STATE{ \qq \ze }
\STATE\label{line5.14}{ \qq \qq $v := 0$; }
\STATE{ \qq \zi \qb{not} \qt{areInv}( $a_{t-1},j$ ) \qb{and} 
                                    ( $t < \ell$ \qb{or} $j \ne aoi$ ) \zt \za }
\STATE{ \qq\qq $a_t := j$; }
\STATE\label{line5.17}{ \qq\qq \zi $u = v$ \zt }
\STATE{ \qq\qq\qq \za \zi \qt{checkInv}( $t,u+1$ ) $<0$ \zt 
                                             \qt{genBrace}( $t+1,p,u,v$ ); \zz;}
\STATE{ \qq\qq \ze \zi $u > v$ \zt }
\STATE\label{line5.20}{ \qq\qq\qq \qt{genBrace}( $t+1,p,u,v$ ); }
\STATE{ \qq \zz; }
\STATE\label{line5.22}{ \qq \zi $u = t$ \zt $u := u-1$; }
\STATE{ \qq \qb{for} $j$ \qb{from} $a_{t-p}+1$ \qb{to} $k-1$ \qb{do} \za }
\STATE{ \qq\qq \zi \qb{not} \qt{areInv}( $a_{t-1},j$ ) \qb{and} 
                                    ( $t < \ell$ \qb{or} $j \ne aoi$ ) \zt \za }
\STATE\label{line5.25}{ \qq\qq\qq \zi $j = aoi$ \zt $v := vv+1$; \ze $v := 0$; }
\STATE{ \qq\qq\qq $a_t := j$; }
\STATE\label{line5.27}{ \qq\qq\qq \zi $u = v$ \zt }
\STATE{ \qq\qq\qq\qq \za \zi \qt{checkInv}( $t,u+1$ ) $<0$ \zt 
                                             \qt{genBrace}( $t+1,t,u,v$ ); \zz;}
\STATE{ \qq\qq\qq \ze \zi $u > v$ \zt }
\STATE\label{line5.30}{ \qq\qq\qq\qq \qt{genBrace}( $t+1,t,u,v$ ); }
\STATE{ \qq\qq \zz; }
\STATE{ \qq \zz; }
\STATE{ \qq \zr; }
\STATE{ \zz; }
\end{algorithmic}
\end{algorithm}

The recursive calls to \qt{genBrace()} are inside the \qt{if}-statements of
  \crefrange{line5.17}{line5.20} and \crefrange{line5.27}{line5.30}.
If $u>v$ then the current prenecklace is less than its inverse, so we can 
  immediately call \qt{genBrace()}.
If $u<v$ then the inverse is less than the prenecklace, the prenecklace cannot
  yield a bracelet, so we do nothing.
If $u=v$ then we need to call the \qt{checkInv()} function to compare the 
  prenecklace with its inverse.
If the prenecklace is less we call \qt{genBrace()}, otherwise we do nothing.

When the \qt{checkInv()} function is called, the current prenecklace starts with
  $u$ copies of $a_1$ and ends with $v=u$ copies of $aoi$.
The remainder, $\gamma$, is non-empty, does not start with $a_1$ or end with
  $aoi$, and is freely reduced.
The arguments $t$ and $i$ are the current prenecklace length and the index of
  the start of $\gamma$.
The \qt{for}-loop compares $\gamma$ with its inverse, returning $-1$ if
  $\gamma$ (and thus the prenecklace) precedes its inverse and $+1$ if 
  $\gamma^{-1}$ precedes $\gamma$.
The upper limit on the \qt{for}-loop is simply a convenient placeholder -- the
  loop is guaranteed to return $-1$ or $+1$ for some 
  $i \leqslant \lfloor (t+1)/2 \rfloor$.

\begin{algorithm}[t]
\caption{The bracelet code wrapper}
\label{alg:a6}
\begin{algorithmic}[1]
\setcounter{ALC@unique}{0}     
\STATE{ \qb{global} $\ell,g,k,aoi$ : integer; \ $a$ : integer array; }
\STATE{ \qb{local} $j$ : integer; }
\STATE{ $g :=$ number of generators; \ $\ell :=$ word length; }
\STATE{ $k := 2g$; }
\STATE{ \qb{for} $j$ \qb{from} $0$ \qb{to} $k-2$ \qb{step} $2$ \qb{do} \za }
\STATE{ \qq $a_1 := j$; }
\STATE{ \qq $aoi :=$ \qt{getInv}( $j$ ); }
\STATE{ \qq \qt{genBrace}( 2,1,1,0 ); }
\STATE{ \zz; }
\end{algorithmic}
\end{algorithm}

\begin{algorithm}[t]
\caption{The \qt{checkInv()} function}
\label{alg:a7}
\begin{algorithmic}[1]
\setcounter{ALC@unique}{0}     
\STATE{ \qb{function} \qt{checkInv}( $t,i$ : integer ) \qb{returns} integer; }
\STATE{ \qb{local} $j$ : integer; }
\STATE{ \za }
\STATE{ \qq \qb{for} $j$ \qb{from} $i$ \qb{to} $t$ \qb{do} \za }
\STATE{ \qq\qq \zi $a_j < \qt{getInv}(\ a_{t-j+1}\ )$ \zt \zr $-1$; }
\STATE{ \qq\qq \zi $a_j > \qt{getInv}(\ a_{t-j+1}\ )$ \zt \zr $+1$; }
\STATE{ \qq \zz;}
\STATE{ \qq \zr 0; }
\STATE{ \zz; }
\end{algorithmic}
\end{algorithm}

\section{Concluding Remarks}\label{sec:discuss}

We have implemented our algorithms in the \qt{C} and \qs{Magma} languages and
  incorporated them into programs for generating and testing lists of
  conjugacy classes and presentations.
The generation of word lists has proved very fast, with the programs' running 
  times being dominated by the times to process the necklaces and bracelets in
  the lists.
Empirical evidence (see \Cref{sup:comp}) suggests that our algorithms are
  CAT\kern-1pt, but we have no proof of this.

The \qt{genNeck()} and \qt{genBrace()} procedures as given process all reduced
  necklaces and bracelets.
If the ``$\ell \bmod p = 0$'' tests are replaced by ``$\ell = p$'' then only the
  reduced aperiodic necklaces (Lyndon words) and bracelets are processed by
  these procedures.

The recursive nature of our necklace and bracelet algorithms induces a tree
  structure on their search spaces.
These trees can easily be split into subtrees, allowing an enumeration to be 
  parallelised \cite{CRSSM00} or distributed across a set of heterogeneous
  machines.

\appendix

\section{Complexity Tests}\label{sup:comp}%
Implementations of our necklace and bracelet listing algorithms for
  reduced words in groups have proved very effective.
However we have no proof that they run in constant amortized time.
Accordingly, in a similar manner to \cite[\S5 \& \S6.2.5]{Saw02}, we produced 
  experimental results for the amount of work done compared with the number of
  necklaces or bracelets generated. 
For these tests the ``process \dots\ $a_1 \cdots a_\ell$'' actions in 
  \qt{genNeck()} and \qt{genBrace()} were replaced by code to
  accumulate the total number of necklaces and of bracelets.
These counts were checked against the expected counts from \Cref{sec:back} and
  used to calculate the work per necklace and per bracelet data.

\begin{figure}[t]
  \centering
  \includegraphics{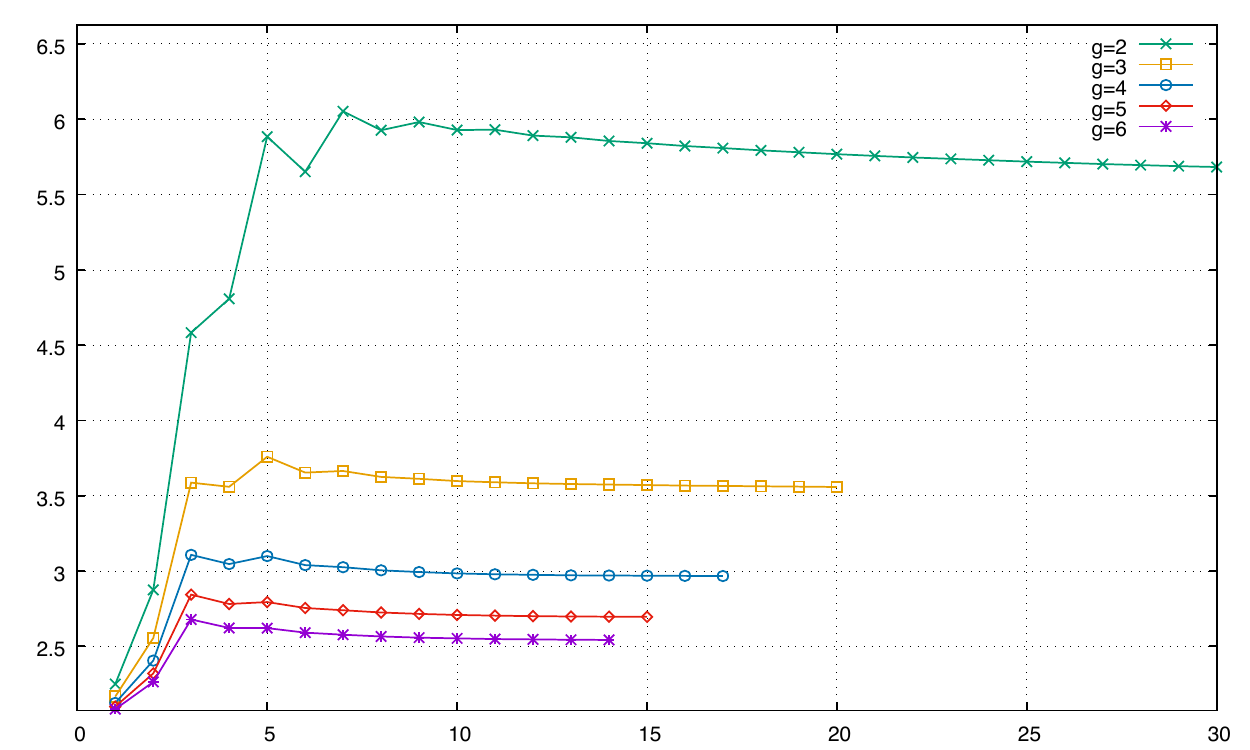}
  \caption{Work per Necklace versus Word Length.}
  \label{fig:nc}
\end{figure}

\begin{figure}[ht]
  \centering
  \includegraphics{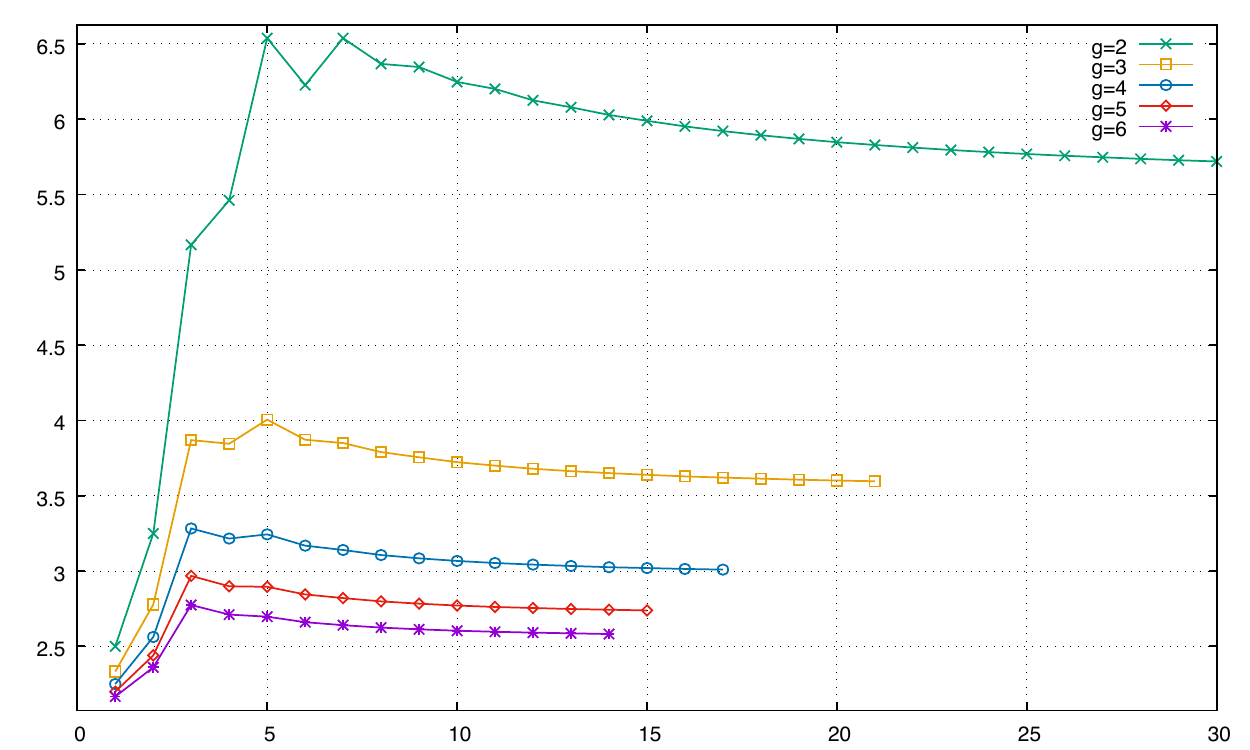}
  \caption{Work per Bracelet versus Word Length.}
  \label{fig:nb}
\end{figure}

The \qt{areInv()} and \qt{getInv()} functions are used by both
  algorithms and are constant time.
Each algorithm starts with a \qt{for}-loop, each iteration of which makes a
  call to \qt{genNeck()} or to \qt{genBrace()}.
For necklaces we count the total number of calls, both direct and recursive, to
  \qt{genNeck()}.
Apart from its embedded \qt{for}-loop, each call to \qt{genNeck()} is 
  constant time.
So our measure of the amount of work done is the total number of calls to 
  \qt{genNeck()} plus the total number of iterations of the embedded
  \qt{for}-loop across all calls to \qt{genNeck()}.
\Cref{fig:nc} plots, for various values of the number of group generators $g$, 
  the word length against the ratio of total work to number of necklaces.

For bracelets we count the total number of calls to \qt{genBrace()} and the
  total number of iterations of its embedded \qt{for}-loop.
We also need to account for the \qt{checkInv()} function, which is not
  constant time.
This function is only called if $u$ and $v$ (in \qt{genBrace()}) are equal, so 
  at least one iteration of \qt{checkInv()}'s embedded \qt{for}-loop is 
  guaranteed for each call.
Thus, we count the total number of iterations of this loop across all calls to
  \qt{checkInv()} and add this to our total.
\Cref{fig:nb} plots, for various values of the number of group generators $g$, 
  the word length against the ratio of total work to number of bracelets.

For both necklaces and bracelets, for all $2 \leqslant g \leqslant 6$, the 
  ratio of the total amount of work done to the number of reduced words 
  generated is decreasing (after an initial peak) as the word length increases.
This strongly suggests that the algorithms are CAT.

\end{document}